\documentclass[11pt]{article}
\usepackage[a4paper,margin=28mm]{geometry}
\usepackage[T1]{fontenc}
\usepackage{amsmath,amssymb,amsthm,mathtools}
\usepackage{booktabs}
\usepackage{array}
\usepackage{graphicx}
\usepackage{microtype}
\usepackage{enumitem}
\usepackage[hidelinks]{hyperref}
\usepackage{cleveref}
\crefname{enumi}{item}{items}
\Crefname{enumi}{Item}{Items}

\newtheorem{theorem}{Theorem}[section]
\newtheorem{proposition}[theorem]{Proposition}
\newtheorem{lemma}[theorem]{Lemma}

\theoremstyle{definition}
\newtheorem{definition}[theorem]{Definition}
\theoremstyle{remark}

\newtheorem*{remark*}{Remark}

\newcommand{\R}{\mathbb{R}}
\newcommand{\Z}{\mathbb{Z}}
\newcommand{\Sone}{\mathbb{S}^{1}}
\newcommand{\CR}{\operatorname{CR}}
\newcommand{\conv}{\operatorname{conv}}
\newcommand{\Var}{\operatorname{Var}}
\newcommand{\osc}{\operatorname{osc}}
\newcommand{\fan}{\operatorname{fan}}
\newcommand{\clip}{\operatorname{clip}}
\newcommand{\RF}{\mathit{RF}}
\newcommand{\pfin}{\varphi_{\mathrm{fin}}}
\newcommand{\Csp}{C_{\mathrm{sp}}}
\newcommand{\dd}{\,\mathrm{d}}
\DeclareUrlCommand\lean{\urlstyle{tt}}
\DeclareUrlCommand\file{\urlstyle{tt}}
\makeatletter\g@addto@macro\UrlBreaks{\do\_\do\-}\makeatother

\title{The logarithmic spiral is optimal for shoreline search\\[2pt]
\large A computer-assisted proof\thanks{Large language models were used throughout this
work---GPT~6 Astra, Claude Fable~5.1, and Claude Opus~5---to find the reduction, to search for the
storage function, and to write the accompanying Lean development and the two verifiers. The storage function's inequalities are verified in Arb ball arithmetic by
two independent implementations, and the finite-dimensional and one-dimensional measure-theoretic
steps of the reduction are formalized in Lean~4; the complete verification development accompanies
this paper as ancillary files. \Cref{sec:checks} states what is machine-checked and what is not.}}
\author{Alexander Temerev\\University of Geneva\\\texttt{alexander.temerev@unige.ch}}
\date{21 September 2026}

\begin{document}
\maketitle

\begin{abstract}
A ship starts at a point of the plane and moves at unit speed; it has to reach an unknown straight
line, of which neither the distance nor the direction is known. The competitive ratio of a path is
the supremum, over all lines, of the time at which the line is reached divided by its distance.
Baeza-Yates, Culberson and Rawlins conjectured that a logarithmic spiral, with ratio
$\Csp=13.8111351794611\ldots$, is optimal. We give a computer-assisted proof. Paths are arbitrary:
the distance from the start and the polar angle may both decrease. The proof lifts the set of
found directions to the universal cover of the circle, where unfolding the polar angle can only
increase it (Kneser--Poulsen on the line); a bookkeeping inequality with a monotone final source
then bounds the covered measure by the reward of a three-state relaxed control problem, in which
inward motion is an ordinary control and excursions below the guaranteed disk are impulses. An
explicit $C^1$ storage function, a tensor cubic B-spline plus a closed-form term, satisfies the
dissipation inequalities of that problem at the spiral's level and is tight only at the spiral; this is
verified with about $10^6$ boxes of Arb ball arithmetic, an exact jet and an interval Hessian at the
spiral.
\end{abstract}


\section{Introduction}\label{sec:intro}

\paragraph{The problem.} A search path is a $1$-Lipschitz map $x\colon[0,\infty)\to\R^2$ with $x(0)=0$.
A shoreline is a line $L(n,d)=\{y:\langle y,n\rangle=d\}$ with unit normal $n\in\Sone$ and distance
$d>0$. The path finds it at time $T_x(n,d)=\inf\{t:\langle x(t),n\rangle\ge d\}$, and
\[
  \CR(x)=\sup_{n\in\Sone,\;d>0}\frac{T_x(n,d)}{d},\qquad C_*=\inf_x\CR(x).
\]
When the distance is known, the optimal path was found by Isbell and proved optimal by Joris
\cite{isbell1957,joris1980}; its length, $1+\sqrt3+7\pi/6=6.39724\ldots$, is a lower bound for the
problem with unknown distance as well. When the distance is not known, Baeza-Yates, Culberson and
Rawlins \cite{baezayates1993} proposed a logarithmic spiral and conjectured that it is optimal;
Finch and Zhu \cite{finchzhu2016} computed its ratio, $\Csp=13.81113\ldots$. Exponential and spiral
strategies are optimal for several related minimax search problems \cite{galchazan1976}, and spiral
search is optimal for the different problem of finding a \emph{point} \cite{langetepe2010}; neither
settles the question here, because for a line the convex hull of the path matters and a ship may
profit from moving inwards. Until recently the strongest unconditional lower bound was Isbell's
$6.3972\ldots$ \cite{baezayatesschott1995}, with $12.5385\ldots$ known only under a cyclic
(spiral-like) restriction, where it is inherited from the axis-parallel shoreline problem
\cite{langetepe2012}; see \cite{georgiou2026} for a recent account. Earlier in this project the
unconditional bound was raised to $C_{\log}=12.5937\ldots$ \cite{temerev2026lower}, by a scheduling
argument unrelated to the method below, and Koch \cite{koch2026cells} showed that $C_*$ is
effectively computable.

\begin{theorem}[Main theorem]\label{thm:main}
Every search path has $\CR(x)\ge\Csp$. Hence $C_*=\Csp$, where $\Csp=1/m_*$ and $(k_*,m_*)$ is
the solution of the two equations of \cref{def:constants},
\[
  \Csp=13.811135179461131253731\ldots,\qquad k_*=4.70655656626900428765\ldots
\]
\end{theorem}

\begin{figure}[t]
\centering
\includegraphics[width=0.62\linewidth]{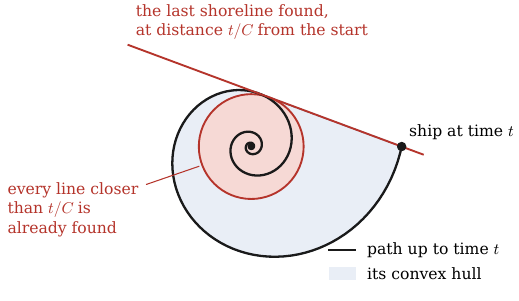}
\caption{The optimal spiral at time $t$. All shorelines at distance at most $t/\Csp$ have been found:
the disk of that radius lies in the convex hull of the path, and touches its boundary. The
supporting line there meets the hull at the ship and again one turn back; that double contact is
the contact equation \eqref{eq:contact} of \cref{def:constants}, and it is why the ship is
reaching this shoreline at exactly time $t$.}
\label{fig:spiral}
\end{figure}

\paragraph{Why decreasing radius is the difficulty.} A chord between two points at distance $R$
from the start and angular distance $\Delta$ has length $2R\sin(\Delta/2)<R\Delta$, and the interior of
a chord adds nothing to the convex hull. A chord is therefore a repositioning which is strictly
cheaper than any path of monotone radius, and no comparison with monotone paths is possible.
The proof below does not try to remove inward motion. It prices it.

\paragraph{Outline.} Write $m=1/C$. A path is competitive at level $m$ if at every time $t$ every
shoreline at distance $mt$ has been found.
\begin{enumerate}[leftmargin=2em,itemsep=1pt]
\item A competitive path can be replaced by a polygonal one, at an arbitrarily small loss
(\cref{sec:polygonal}).
\item A point at distance $R\ge D$ finds the shorelines at distance $D$ whose normals lie in an arc of
half-width $\arccos(D/R)$ about its polar angle. If these arcs cover the circle, the corresponding
intervals on the line, centred at \emph{any} lifts of the angles, have total length at least $2\pi$; and
replacing the lifted angle by its total variation (``unfolding'') only lengthens their union, by the
Kneser--Poulsen theorem on the line. So coverage forces the unfolded measure $\mu(\tau)\ge2\pi$, and
$\mu_{m_*}\ge2\pi+\mathrm{const}\cdot\log(m/m_*)$ for a path which is competitive at a level $m>m_*$
(\cref{sec:lift}).
\item In logarithmic coordinates every direction is covered from its first coverage until the
expiry of its best source. Because $\log\cos$ is concave, the best source is a monotone function of
the direction. This bounds $\int\mu\dd\tau$ by a \emph{reward}: a fan term paid when directions
are finalised, plus the distance between two monotone fronts (\cref{sec:book}).
\item The reward is a functional of a three-dimensional state $(v,a,c)$ whose dynamics are
controlled by the heading of the ship. One of the fronts is a running maximum; the Skorokhod
reflection lemma gives its dynamics (\cref{sec:traj}).
\item If a $C^1$ function $V$ satisfies four inequalities, then $V+\text{reward}$ does not increase along
any trajectory, and the mean of $\mu_{m_*}-2\pi$ is at most zero: a contradiction with step~2
(\cref{sec:verif,sec:proof}).
\item Such a $V$ is exhibited and verified rigorously (\cref{sec:cert}).
\end{enumerate}
The main text carries this line of argument and nothing else. Four appendices hold the technical
material it rests on: the measure-theoretic properties of the lift (\cref{app:lift}), the regularity
of the state (\cref{app:reg}), the details of the certificate (\cref{app:cert}), and a numerical
test of the whole reduction on explicit paths (\cref{app:test}).

The tools come from several fields: Kneser--Poulsen \cite{poulsen1954,kneser1955,bezdekconnelly2002},
monotone comparative statics and Monge arrays \cite{topkis1978,milgromshannon1994,aggarwal1987},
the Skorokhod problem \cite{skorokhod1961,lionssznitman1984}, dissipation inequalities and the
convex-duality view of optimal control \cite{willems1972,vinter1993}, and ball arithmetic
\cite{johansson2017arb}. The shape of the certificate --- a smooth strict subsolution which is tight
on one orbit --- is what weak KAM theory predicts for Tonelli systems \cite{fathisiconolfi2004,bernard2007};
our system is not of that type, and no theorem of that theory is used.

\section{Setting and the spiral constant}\label{sec:setting}

For a path $x$ put $h_t(n)=\max_{s\le t}\langle x(s),n\rangle$, the support function of
$\conv x([0,t])$ \cite{schneider2014}. Since $x$ is continuous and $x(0)=0$, the
line $L(n,d)$ has been found by time $t$ exactly when $h_t(n)\ge d$. Hence:

\begin{proposition}\label{prop:hull}
$\CR(x)\le C$ if and only if $h_t(n)\ge t/C$ for every $t>0$ and every $n\in\Sone$, that is, if and only
if the disk of radius $t/C$ about the start lies in $\conv x([0,t])$ for every $t$.
\end{proposition}

\begin{definition}\label{def:competitive}
A path is \emph{competitive at level $m$} if $h_t(n)\ge mt$ for all $t>0$ and $n\in\Sone$.
\end{definition}

The arguments below use this condition only for large $t$, so they also apply if distances are
restricted to $d\ge1$, or if an additive constant is allowed in the definition of the ratio.

\paragraph{Notation.} Throughout, $m=1/C$ is a level, $h(s)=\arccos(e^{-s})$ for $s\ge0$, and for a
point of the path at time $t$ and distance $R$ from the start
\[
 \tau=\log t,\qquad z=\log R,\qquad v=z-\tau-\log m=\log\frac{R}{mt},\qquad P=\frac tR=\frac{e^{-v}}{m}.
\]
The point is outside the guaranteed disk of radius $mt$ exactly when $v\ge0$, and
$v\le\log(1/m)$ because $R\le t$. We write $x_+=\max(x,0)$.

\begin{definition}[Spiral constants]\label{def:constants}
For $k>0$ and $0<m<(1+k^2)^{-1/2}$ let
\[
 \alpha=\arctan\frac1k,\qquad \beta=\arccos\bigl(m\sqrt{1+k^2}\bigr),\qquad
 \delta=\log\frac{k}{m(1+k^2)} .
\]
The pair $(k_*,m_*)$ is the unique solution, in the box $|k-4.7066|\le0.05$, $|m-0.072405|\le5\cdot10^{-4}$, of
\begin{align}
  \beta+k\,\delta+\alpha&=2\pi, &&\text{(contact)}\label{eq:contact}\\
  \delta\,(1+k^2)&=k\,(k+\cot\beta). &&\text{(optimal pitch)}\label{eq:pitch}
\end{align}
We write $\alpha_*,\beta_*,\delta_*$ for the values at $(k_*,m_*)$, $P_*=\sqrt{1+k_*^2}$,
$\kappa=m_*/\sqrt{1-m_*^2}$, and $\Csp=1/m_*$.
\end{definition}

Existence and uniqueness in the box are certified by an interval Newton step
(\file{scripts/check_spiral_constants_box.py});
$m_*=0.0724053444561982\ldots$, $\alpha_*=0.2093562673\ldots$, $\beta_*=1.2149466769\ldots$.

\begin{proposition}[The spiral]\label{prop:spiral}
The logarithmic spiral $R=e^{\theta/k}$, traversed at unit speed, is competitive at level $m$ if and
only if $\beta+k\delta+\alpha\ge2\pi$. The largest such level, over all $k$, is $m_*$, attained at $k_*$.
In particular $C_*\le\Csp$.
\end{proposition}

\begin{proof}
On the spiral $t=\sqrt{1+k^2}\,R$, so every point has the same $v$, with $e^{-v}=m\sqrt{1+k^2}=\cos\beta$.
By \cref{sec:lift}, the point at polar angle $\theta-u$, $u\ge0$, finds at time $t$ the directions
$\varphi$ with $|\varphi-(\theta-u)|\le h(v-u/k)$, as long as $u\le kv$. The polar angle is increasing, so
these intervals of the line form an interval with right end $\theta+\max_u[h(v-u/k)-u]=\theta+\beta$ and left
end $\theta-\max_u[u+h(v-u/k)]$. Since $h'=\cot h$, the last maximum is attained where $h=\alpha$,
that is at $u=k\delta$ with $\delta=v+\log\cos\alpha$, and equals $k\delta+\alpha$. An interval of the line
covers the circle if and only if its length is at least $2\pi$. The length $\beta+k\delta+\alpha$
decreases in $m$; along \eqref{eq:contact}, $\mathrm{d}m/\mathrm{d}k=0$ is equivalent to
$\partial_k(\beta+k\delta+\alpha)=\delta-k(k+\cot\beta)/(1+k^2)=0$, which is \eqref{eq:pitch}.
That this critical point is the maximum of $m$ over $k$ is classical \cite{baezayates1993,finchzhu2016}.
\end{proof}

Only the inequality $C_*\ge\Csp$ needs proof. The two equations reappear in \cref{lem:jet}: they
are exactly what makes the certificate tight at the spiral.

\section{Reduction to polygonal paths}\label{sec:polygonal}

\begin{lemma}\label{lem:polygonal}
Let $x$ be competitive at level $m$ and $0<\eta<m$. Put $t_j=(1+\eta)^j$, $j\in\Z$, and let $y$ be the
path with $y(t_j)=x(t_j)$ which runs along each chord $[x(t_j),x(t_{j+1})]$ at constant speed. Then
$y$ is a search path, it has finitely many vertices in every compact interval of $(0,\infty)$, and it is
competitive at level $m_\eta=(m-\eta)/(1+\eta)$.
\end{lemma}

\begin{proof}
A chord is not longer than $t_{j+1}-t_j$, so $y$ is $1$-Lipschitz, and $|y(t)|\le t$ gives continuity
at $0$. Let $t_i\le t<t_{i+1}$ and $0<s\le t_i$, say $t_j\le s\le t_{j+1}$ with $j<i$. Then
\[
 \langle x(s),n\rangle\le\langle x(t_j),n\rangle+(s-t_j)\le h^y_t(n)+\eta\,t_j\le h^y_t(n)+\eta\,t_i ,
\]
and $\langle x(0),n\rangle=0\le h^y_t(n)$. Hence
$h^y_t(n)\ge h^x_{t_i}(n)-\eta t_i\ge(m-\eta)t_i\ge m_\eta t$.
\end{proof}

This is \lean{Shore.polygonal_reduction}, for $1$-Lipschitz paths in any real inner product space.
Since $m_\eta\to m$, a path which is competitive at a level above $m_*$ yields a polygonal one with
the same property. \textbf{From now on paths are polygonal, with vertex times $t_j$.} Their speed
may be below one.

\section{From planar coverage to lifted measure}\label{sec:lift}

Fix a query $(t,D)$ with $D=mt$. A point with polar coordinates $(R,\theta)$ lies on or beyond the line
$L(n_\varphi,D)$, $n_\varphi=(\cos\varphi,\sin\varphi)$, if and only if $R\cos(\varphi-\theta)\ge D$. So a point
with $R\ge D$ (a \emph{live source}) finds the closed arc of directions
\[
  \{\varphi:\ |\varphi-\theta|\le\arccos(D/R)\}\pmod{2\pi},
\]
its \emph{cap}, and a point with $R<D$ finds nothing.

\begin{lemma}[Vertices suffice]\label{lem:vertices}
For a polygonal path the set of directions found at time $t$ is the union of the caps of the vertices
visited so far and of the current point. Only finitely many of them are live.
\end{lemma}
\begin{proof}
A linear function on a segment is largest at an endpoint (\lean{Shore.segment_support_le_max}).
Since $R(s)\le s$, a live vertex has $t_j\ge mt$.
\end{proof}

\begin{lemma}[Lifting and unfolding]\label{lem:KP}
Let finitely many live sources $(R_i,\theta_i)$ find every direction, let $\tilde\theta_i\in\R$ be arbitrary
lifts of their polar angles, and let $\Theta_i\in\R$ satisfy $|\tilde\theta_i-\tilde\theta_j|\le|\Theta_i-\Theta_j|$
for all $i,j$. Then
\[
  \Bigl|\,\bigcup_i\bigl[\Theta_i-\arccos(D/R_i),\ \Theta_i+\arccos(D/R_i)\bigr]\Bigr|\ \ge\ 2\pi .
\]
\end{lemma}
\begin{proof}
A cap is the projection of the interval $[\tilde\theta_i\pm\arccos(D/R_i)]$. The $2\pi$-translates of
the union $U$ of these intervals cover the line, so $|U|\ge2\pi$: the projection to the circle does
not increase measure. Finally, if the centres of finitely many intervals are moved so that no
pairwise distance decreases, the length of the union does not decrease. This is the
one-dimensional Kneser--Poulsen theorem, for arbitrary radii: if the new union has a gap, split the
family there and use induction; otherwise both unions are compared with the hull length
$\max_{i,j}(|x_i-x_j|+r_i+r_j)$, which is monotone in the distances.
\end{proof}
The three steps are \lean{exists_int_abs_sub_le_arccos}, \lean{volume_ge_of_periodic_cover} and
\lean{kneser_poulsen_line}; together, \lean{Shore.unfolded_measure_ge_of_planar_cover}.

\paragraph{The unfolding.} Along a polygonal path choose a lift $\tilde\theta$ of the polar angle
which is continuous while $R>0$, right-continuous, and jumps --- only at a passage through the
origin or a departure from a wait there --- by at most $\pi$ in absolute value; \cref{app:lift}
gives the construction and checks that $\tilde\theta$ has locally bounded variation. Let $\Theta$ be
its signed variation function relative to a fixed $s_0>0$,
\[
 \Theta(s)=\Var(\tilde\theta;[s_0,s])\ \ (s\ge s_0),\qquad \Theta(s)=-\Var(\tilde\theta;[s,s_0])\ \ (s<s_0).
\]
Radius, speed and clock are unchanged, and $|\tilde\theta(s)-\tilde\theta(u)|\le|\Theta(s)-\Theta(u)|$
(\lean{Shore.unfolding_is_expansion}), so \cref{lem:KP} applies to the live vertices with
$\Theta_i=\Theta(s_i)$; enlarging the family to all live points of the trajectory only enlarges the
union.

\begin{lemma}[Where $\Theta$ jumps]\label{lem:jumps}
$\Theta$ is nondecreasing and right-continuous, and it is discontinuous only at times with $R=0$.
Every jump has size at most $\pi$.
\end{lemma}
\begin{proof}
$\tilde\theta$ is continuous wherever $R>0$, so its variation $\Theta$ is continuous there too; and
$\Theta$ jumps by the absolute value of the jump of $\tilde\theta$.
\end{proof}

Every later step which meets a jump of $\Theta$ uses \cref{lem:jumps} in the same way: at a jump the
radius is zero, so the point covers nothing and is not a source. Note also that a zero-radius time
is always strictly inside a \emph{dive} (\cref{sec:traj}). With
\[
 \mu_m(\tau)=\Bigl|\bigcup_{s\le t,\ R(s)\ge mt}\bigl[\Theta(s)-\arccos\tfrac{mt}{R(s)},\ \Theta(s)+\arccos\tfrac{mt}{R(s)}\bigr]\Bigr|
\]
we have proved:

\begin{proposition}\label{prop:mu}
If a polygonal path has found every shoreline at distance $mt$ by time $t=e^\tau$, then
$\mu_m(\tau)\ge2\pi$. No assumption is made on the radius or on the angle.
\end{proposition}

A live source has $mt\le R(s)\le s\le t$, so all live sources lie in the compact time interval
$[mt,t]$ and $\mu_m(\tau)\le\Theta(t)-\Theta(mt)+\pi<\infty$. That finiteness --- the only role of the
restriction $R(s)\ge mt$ --- is what makes the boundary term of \cref{thm:book} a constant in
\cref{sec:proof}.

\begin{lemma}[Strictness]\label{lem:strict}
If a polygonal path is competitive at a level $m>m_*$, then for every $\tau$
\[
  \mu_{m_*}(\tau)\ \ge\ 2\pi+2\kappa\log(m/m_*),\qquad \kappa=m_*/\sqrt{1-m_*^2}.
\]
\end{lemma}
\begin{proof}
Lowering the level from $m$ to $m_*$ replaces the half-width $h(s)$, $s=\log(R/(mt))\ge0$, of every
live source by $h(s+\lambda)$, $\lambda=\log(m/m_*)$, and $s+\lambda\le\log(1/m_*)$ because $R\le t$. On
$[0,\log(1/m_*)]$ we have $h'=\cot h\ge\cot\arccos m_*=\kappa$, so every half-width grows by at least
$\kappa\lambda$ (\lean{Shore.cap_halfwidth_gain}). Enlarging all radii of a finite nonempty family of
intervals by $\epsilon$ enlarges the union by at least $2\epsilon$, because the two outermost ends move
outwards (\lean{Shore.union_widen_of_le}). Apply this to the finite family of \cref{lem:KP}.
\end{proof}

\section{The bookkeeping inequality}\label{sec:book}

In logarithmic form, the source visited at time $s$ (we also write $s$ for its logarithm when no
confusion can arise) covers the lifted direction $\varphi$ at the query time $\tau$ if and only if
\[
  \tau+\log m\ \le\ z(s)+\log\cos(\varphi-\Theta(s)),\qquad |\varphi-\Theta(s)|<\tfrac\pi2,\qquad s\le\tau .
\]
The left side rises at unit speed: a source covers a direction from its visit until its \emph{score}
$z+\log\cos(\varphi-\Theta)$ is overtaken by the query level. Give the score the value $-\infty$
when $R(s)=0$ or $|\varphi-\Theta(s)|\ge\pi/2$. Fix $0<m<1$ and a window $[\tau_0,\tau_1]$.
If the path has stayed at the origin up to $e^{\tau_1}$, its covered measure is zero and there is
nothing to estimate; otherwise define, using the sources visited by $\tau_1$,
\begin{itemize}[leftmargin=2em,itemsep=1pt]
\item the \emph{final envelope} $E(\varphi)=\sup_{s\le\tau_1}\bigl[z(s)+\log\cos(\varphi-\Theta(s))\bigr]$,
and its domain $\Omega=\{\varphi:E(\varphi)>-\infty\}$;
\item for $\varphi\in\Omega$, the \emph{final source} $\sigma(\varphi)$: the maximiser with the largest angle, and among maximisers of
equal angle the one visited first;
\item the \emph{expiry} $\tau_e(\varphi)=E(\varphi)-\log m$, and the first time $\tau_f(\varphi)$ at which $\varphi$ is covered
\end{itemize}
($\tau_f=+\infty$ if $\varphi$ is never covered). The supremum is attained and both tie-breaking
choices are legitimate: this is \cref{lem:attained}, where the right-continuity of $\Theta$, and not
its continuity, is what does the work. Write $\tau_\sigma(\varphi)$ for the visit time of
$\sigma(\varphi)$, and $v_\sigma$, $a_\sigma=\Theta(\tau_\sigma)-\varphi$ for its data.
All of $E$, $\Omega$, $\sigma$, $\tau_\sigma$ and the front $\pfin$ below depend on the right end
$\tau_1$ of the window; $\RF$ does not. Everything is for one fixed window, and the bound of
\cref{thm:A} is uniform in it, which is what \cref{sec:proof} uses.

All direction integrals below are over $\Omega$; directions outside it contribute no covered time.
Put $L=\inf\Omega\in\R\cup\{-\infty\}$ and $U=\sup\Omega\in\R$. By \cref{lem:omega}, $\Omega$ is an
interval up to a countable set: $(L,U)\setminus\Omega$ is countable, and $U=\Theta(\tau_1)+\pi/2$
when $R(e^{\tau_1})>0$.

\begin{lemma}[Monotone final source]\label{lem:selector}
$\tau_\sigma\colon\Omega\to\R$ is nondecreasing.
\end{lemma}
\begin{proof}
The kernel $f=\log\cos$ is concave on $(-\pi/2,\pi/2)$. For $\varphi_1\le\varphi_2$ and
$\Theta_2\le\Theta_1$ the numbers $\varphi_1-\Theta_2$ and $\varphi_2-\Theta_1$ lie between
$\varphi_1-\Theta_1$ and $\varphi_2-\Theta_2$ and have the same sum, so
$f(\varphi_1-\Theta_1)+f(\varphi_2-\Theta_2)\le f(\varphi_1-\Theta_2)+f(\varphi_2-\Theta_1)$. Therefore, if the
best sources of $\varphi_1\le\varphi_2$ are crossed, each of them is a best source of both
directions, and the largest-angle best source has an angle which is monotone in $\varphi$
(\lean{Shore.finalSource_angle_mono}). This is the exchange argument of monotone comparative
statics \cite{topkis1978,milgromshannon1994}. Now for the visit times. If the angles of $\sigma(\varphi_1)$
and $\sigma(\varphi_2)$ differ, the smaller angle is visited earlier, because $\Theta$ is nondecreasing in
time. If they are equal, the maximisers with that angle are the times with that angle and the largest
$z$ among them, a set which does not depend on the direction; so the earliest one is the same.
\end{proof}

Define the \emph{finalisation front} and the \emph{first-coverage front}
\[
  \pfin(\tau)=\sup\bigl(\{L\}\cup\{\varphi\in\Omega:\tau_\sigma(\varphi)\le\tau\}\bigr),\qquad
  \RF(\tau)=\sup_{s\le\tau}\bigl[\Theta(s)+h(v(s)_+)\bigr].
\]
Both are nondecreasing, and $\pfin$ is finite at every finite $\tau\le\tau_1$, with
$\pfin(\tau)\le\Theta(\tau)+\pi/2$: when $L$ is finite it initializes the front even if nothing has
been finalised yet, and when $L=-\infty$ the directions of $\Omega$ below $\Theta(\tau)-\pi/2$ have
only past sources, so the set is nonempty. For the Stieltjes measure extend $\pfin$ constantly
after $\tau_1$. At the origin set $v=-\infty$, so $h(v_+)=0$.

For $s<t\le\tau_1$ the two sets $\{\varphi\in\Omega:\tau_\sigma\le s\}$ and
$\{\varphi\in\Omega:\tau_\sigma\le t\}$ are, by \cref{lem:selector}, initial intervals of $\Omega$ up to
their endpoints; since $\Omega$ has countable complement in $(L,U)$, their difference has length
$\pfin(t)-\pfin(s)$, that is
\begin{equation}\label{eq:pushforward}
 \bigl|\{\varphi\in\Omega:s<\tau_\sigma(\varphi)\le t\}\bigr|=\pfin(t)-\pfin(s),
\end{equation}
a finite difference even when $L=-\infty$. So the push-forward of Lebesgue measure on $\Omega$ under
$\tau_\sigma$ is the Stieltjes measure $\mathrm{d}\pfin$; continuity from above of that locally
finite measure also shows that $\pfin$ is right-continuous.

\begin{theorem}[Bookkeeping]\label{thm:book}
For every polygonal path, $0<m<1$, and every window with the definitions above,
\[
 \int_{\tau_0}^{\tau_1}\!\mu_m\dd\tau\ \le
 \int_{\{\varphi\in\Omega:\tau_0<\tau_\sigma(\varphi)\le\tau_1\}}\!\bigl[v_\sigma+\log\cos a_\sigma\bigr]_+\dd\varphi
 \;+\;\int_{\tau_0}^{\tau_1}\!\bigl(\RF-\pfin\bigr)_+\dd\tau\;+\;\log\tfrac1m\;\mu_m(\tau_0).
\]
\end{theorem}
\begin{proof}
The covered set $\{(\varphi,\tau):\varphi\text{ is covered at }\tau\}$ is
$\{E_\tau(\varphi)\ge\tau+\log m\}$ with $E_\tau(\varphi)=\sup_{s\le\tau}[z(s)+\log\cos(\varphi-\Theta(s))]$;
it is measurable (\cref{lem:measurable}). By Tonelli's theorem
$\int\mu\dd\tau=\int_\Omega T(\varphi)\dd\varphi$, where $T(\varphi)$ is the measure of the
times in the window at which $\varphi$ is covered. A covered time satisfies $\tau+\log m\le E(\varphi)$,
so the covered times lie in $[\max(\tau_f,\tau_0),\min(\tau_e,\tau_1)]$.

\emph{Case $\tau_\sigma>\tau_0$.} For any intermediate time (\lean{Shore.covered_time_le}),
\[
  T\le(\tau_e-\tau_\sigma)_++\bigl(\tau_\sigma-\max(\tau_f,\tau_0)\bigr)_+ ;
\]
this
also holds if the final source does not cover $\varphi$ at its visit, for then the first term vanishes
and $T\le\tau_e-\tau_f<\tau_\sigma-\tau_f$. The first term is
$\tau_e-\tau_\sigma=z_\sigma+\log\cos a_\sigma-\log m-\tau_\sigma=v_\sigma+\log\cos a_\sigma$. By
Tonelli again, the second term integrates to the measure of
$\{(\varphi,\tau):\max(\tau_f,\tau_0)\le\tau<\tau_\sigma(\varphi)\}$. For fixed $\tau$ its section lies in
$[\pfin(\tau),\RF(\tau)]$: if $\tau_\sigma(\varphi)>\tau$, then $\tau_\sigma(\varphi')>\tau$ for all
$\varphi'>\varphi$ by \cref{lem:selector}, so $\pfin(\tau)\le\varphi$; and a direction covered at a time
$\tau'\le\tau$ by a source $s\le\tau'$ has
$|\varphi-\Theta(s)|\le h(z(s)-\tau'-\log m)\le h(v(s)_+)$, so $\varphi\le\RF(\tau)$.

\emph{Case $\tau_\sigma\le\tau_0$ and $T>0$.} Then $\tau_e\ge\tau_0$, so the final source, which was visited
by $\tau_0$, covers $\varphi$ at $\tau_0$; and
$T\le\tau_e-\tau_\sigma=v_\sigma+\log\cos a_\sigma\le\log(1/m)$ because $R\le t$. Such directions have
measure at most $\mu_m(\tau_0)$. In particular, directions finalised exactly at the initial
endpoint belong to this boundary term, not to the fan integral.
\end{proof}

Lapses of coverage are counted as covered and the front $\RF$ ignores expiry; both only weaken
the bound. On the spiral the inequality is an equality up to the boundary term: it is the contact
equation \eqref{eq:contact}.

\section{The three-state relaxed problem}\label{sec:traj}

Let the \emph{state} of a polygonal path at a time with $v\ge0$ be
\[
  (v,a,c),\qquad a=\Theta-\pfin,\qquad c=\RF-\Theta,\qquad q=a+c=\RF-\pfin .
\]
The front coordinates $a,c$ are also used during dives. Use right-continuous representatives and
write $\mathrm{d}\pfin^{\,c}$ for the continuous part of the Stieltjes measure, including any
singular continuous part. Define $\fan(v,a)=[v+\log\cos a]_+$ for $|a|<\pi/2$, and zero when
$|a|=\pi/2$ or $v=-\infty$. The \emph{fan reward measure} is
\[
 \mathcal F(B)=\int_{\{\varphi\in\Omega:\tau_\sigma(\varphi)\in B\}}
       \fan(v_\sigma,a_\sigma)\dd\varphi .
\]
By \eqref{eq:pushforward}, its continuous part is $\fan(v(\tau),a(\tau))\,\mathrm{d}\pfin^{\,c}(\tau)$.
At an atom $u$ of $\mathrm{d}\pfin$ the directions finalised at $u$ fill $(\pfin(u-),\pfin(u))$ up to
endpoints and a null set --- and $\Theta$ is continuous at such a $u$, because a zero-radius point is
never a final source and all jumps of $\Theta$ occur at zero radius (\cref{lem:jumps}) --- so
\begin{equation}\label{eq:fan-atom}
 \mathcal F(\{u\})=\int_{a(u)}^{a(u-)}\fan(v(u),r)\dd r.
\end{equation}
The shorthand $\fan\,\mathrm{d}\pfin$ below always means $\mathcal F$: at a jump the \emph{whole} fan
is integrated, not an endpoint value times the jump size.

\begin{remark*}[The ray]
For $x(t)=(t,0)$ take $\Theta\equiv0$. Then $\Omega=(-\pi/2,\pi/2)$, every final source is the
endpoint $e^{\tau_1}$, and $\pfin$ is $-\pi/2$ before $\tau_1$ and $\pi/2$ at $\tau_1$; so
$a=\Theta-\pfin$ stays in $[-\pi/2,\pi/2]$, as \cref{lem:state} asserts. Both endpoint fan values at
the jump are zero while the integral in \eqref{eq:fan-atom} is positive: the distinction above is
necessary already here.
\end{remark*}

With these conventions \cref{thm:book} reads
\begin{equation}\label{eq:reward}
 \begin{split}
  \int_{\tau_0}^{\tau_1}\!(\mu_m-2\pi)\dd\tau
    &\le\int_{(\tau_0,\tau_1]}\!\text{reward}+\log\tfrac1m\,\mu_m(\tau_0),\\
  \text{reward}&=\mathcal F+\bigl(q_+-2\pi\bigr)\dd\tau.
 \end{split}
\end{equation}

\begin{lemma}[State space]\label{lem:state}
$|a|\le\pi/2$ and $h(v_+)\le c\le\arccos m$.
\end{lemma}

The proof is in \cref{app:reg}; the delicate point is that $\pfin(\tau)\ge\Theta(\tau)-\pi/2$ holds
also at the zero-radius times which carry the jumps of $\Theta$.

If the ship has speed $\rho\in[0,1]$ and heading $b\in[0,\pi]$ relative to the outward radial
direction (the sign of the angular motion has been unfolded), then
\begin{equation}\label{eq:dyn}
 v'=\rho P\cos b-1,\qquad \Theta'=\rho P\sin b,\qquad a'=\Theta'-\pfin',
\end{equation}
where $'=\mathrm{d}/\mathrm{d}\tau$. A \emph{dive} is a maximal interval on which $v<0$.

\begin{lemma}[Regularity and the Skorokhod step]\label{lem:skorokhod}
On a compact window:
\begin{enumerate}[leftmargin=2em,itemsep=1pt]
\item $\Theta$ is nondecreasing and piecewise analytic, with finitely many upward jumps, all at
passages through the origin and hence strictly inside dives; outside dives $v$ and $\Theta$ are
Lipschitz; there are finitely many dives.
\item $g=\Theta+h(v_+)$ is absolutely continuous between the jumps, and so is its running maximum
$\RF$.
\item For almost every $\tau$ outside dives,
\begin{equation}\label{eq:S}
 \text{either}\quad c>h(v)\ \text{and}\ c'=-\Theta',\qquad\text{or}\qquad c=h(v)\ \text{and}\ c'=\bigl(h(v)\bigr)'.
\end{equation}
\item $a$ has bounded variation; outside dives it jumps only downwards.
\end{enumerate}
\end{lemma}

Again the proof is in \cref{app:reg}. Statement \eqref{eq:S} is the regulator property of the
Skorokhod reflection problem \cite{skorokhod1961,lionssznitman1984}: $c-h(v_+)\ge0$ is the reflected
path, and the regulator $\RF$ increases only on the contact set: no maximum has to be
differentiated, and no sign condition on a derivative of $V$ will be needed. Part~(4) is what makes
hypothesis (A1) below usable, since $V_a\ge\fan$ prices a \emph{decrease} of $a$; an upward jump of
$a$ outside a dive would be fatal.

\section{The verification theorem}\label{sec:verif}

Let $D_3$ be the state space of \cref{lem:state} at the level $m_*$,
\[
  D_3=\{0\le v\le\log\tfrac1{m_*},\ |a|\le\tfrac\pi2,\ h(v)\le c\le\arccos m_*\} ,
\]
a compact set; the constraint $c\ge h(v)$ reads $v\le-\log\cos c$, so the bound $v\le\log\frac1{m_*}$
is implied by $c\le\arccos m_*$. Below, ``$V$ is $C^1$ on $D_3$'' means that $V$ is the restriction of
a $C^1$ function on an open neighbourhood of $D_3$; for the $V$ of \cref{sec:cert} such an extension
is written down in \cref{app:cert}, so no extension theorem is needed. For $0<D\le\pi$ put
\[
  \Delta\tau_{\min}(D)=2\operatorname{arsinh}\bigl(\kappa\sin\tfrac D2\bigr),
\]
which \cref{lem:divetime} identifies as a lower bound for the duration of a dive whose unfolded
angular advance is $D$.

The two dissipation hypotheses have the same shape. With $P=e^{-v}/m_*$ and $q=a+c$ set
\begin{equation}\label{eq:Phi}
  \Phi(X,Y)\ =\ P\sqrt{X^2+Y_+^2}\ -\ X\ -\ 2\pi\ +\ q_+ ,\qquad X,Y\in\R .
\end{equation}
For $b\in[0,\pi]$ one has $X\cos b+Y\sin b\le\sqrt{X^2+Y_+^2}$, with equality for some heading
(\lean{Shore.heading_sup_le}, \lean{heading_sup_attained}); so $\Phi\le0$ says exactly that no
heading and no speed makes the corresponding dissipation rate positive.

\begin{theorem}[Verification]\label{thm:A}
Let $V$ be $C^1$ on $D_3$ and satisfy
\begin{itemize}[leftmargin=3.2em,itemsep=2pt]
\item[\rm(A1)] $V_a\ge\fan(v,a)$ on $D_3$;
\item[\rm(A2)] $\Phi\bigl(V_v,\ V_a-V_c\bigr)\le0$ on $D_3$;
\item[\rm(A3)] $\Phi\bigl(V_v+V_c\,h'(v),\ V_a\bigr)\le0$ on the surface $c=h(v)$, and in the limit
$v\to0$;
\item[\rm(A4)] $V\bigl(0,\min(a+D,\tfrac\pi2),\max(c-D,0)\bigr)-V(0,a,c)\le\bigl(2\pi-Q_+\bigr)\,\Delta\tau_{\min}(D)$
for all $(a,c)$ and $0<D\le\pi$, where $Q=\max(q_0,q_1)$ with $q_0=a+c$ and $q_1$ the same sum at
the arrival state.
\end{itemize}
Then for every polygonal path and every window whose ends have $v\ge0$,
\[
  \int_{\tau_0}^{\tau_1}\bigl(\mu_{m_*}-2\pi\bigr)\dd\tau\ \le\ \osc V+\log\tfrac1{m_*}\,\mu_{m_*}(\tau_0).
\]
\end{theorem}

\noindent Note that (A1) forces $V_a\ge0$: this makes the positive part in $\Phi$ inactive in
(A3), and it is used again in the dive step of the proof.

\begin{proof}
Put $W(\tau)=V(v,a,c)$, with the right-continuous state, at the times with $v\ge0$. We show that
$\mathrm{d}W+\fan\dd\pfin+(q_+-2\pi)\dd\tau\le0$ as measures outside dives, and the corresponding
inequality across each dive. Integrating on $(\tau_0,\tau_1]$ telescopes to
$W(\tau_1)-W(\tau_0)$: the jump at $\tau_0$ is excluded, and the one at $\tau_1$ is included.
With \cref{thm:book} and \eqref{eq:reward} this gives the claim, since
$W(\tau_0)-W(\tau_1)\le\osc V$.

\emph{Outside dives.} Outside dives $v$ and $\Theta$ are Lipschitz and $\RF$ is absolutely
continuous (\cref{lem:skorokhod}), so $\dd v$, $\dd\Theta$ and $\dd c$ have no part singular with
respect to $\dd\tau$. The only measure below which may have one is $\dd\pfin$, and it enters with
the favourable sign; so the pointwise computation that follows is an inequality between measures
and not only between densities. The state is of bounded variation with values in $D_3$ and $V$
extends to a $C^1$ function with bounded gradient, so the chain rule for a $C^1$ function of a
function of bounded variation \cite[Theorem~3.96]{afp2000} gives
\[
 \mathrm{d}W=V_v\dd v+V_a\bigl(\mathrm{d}\Theta-\mathrm{d}\pfin^{\,c}\bigr)+V_c\dd c+[\text{jumps of }a].
\]
A downward jump at $u$, at fixed $(v,c)$, changes $W$ by
\[
 W(u)-W(u-)=-\int_{a(u)}^{a(u-)}V_a(v(u),r,c(u))\dd r
       \le-\mathcal F(\{u\}),
\]
by (A1) and \eqref{eq:fan-atom}; the continuous part of $\mathrm{d}\pfin$ contributes
$-(V_a-\fan)\dd\pfin^{\,c}\le0$ in the same way. For the rest use \eqref{eq:dyn} and \eqref{eq:S}.
Where $c>h(v)$,
\[
 V_vv'+V_a\Theta'+V_cc'+q_+-2\pi=\rho P\bigl[V_v\cos b+(V_a-V_c)\sin b\bigr]-V_v+q_+-2\pi ,
\]
whose supremum over $\rho\in[0,1]$ and $b\in[0,\pi]$ is $\Phi(V_v,V_a-V_c)$: this is (A2). Inward
headings are the half $\cos b<0$; pricing them costs no more than replacing $(V_v)_+$ by $|V_v|$.
Where $c=h(v)$, the same computation with $c'=h'(v)v'$ gives (A3)
(\lean{Shore.lifted_sliding_dissipation}). The forced finalisation at $a=\pi/2$ is an admissible
finalisation rate. At $v\to0$ on the surface $c'$ is unbounded, but $V_ch'(v)$ is bounded for the
$V$ of \cref{sec:cert}, and $V$ is $C^1$ in $(v,a,c)$.

\emph{Dives.} While $v<0$ the point contributes no current coverage, $h(v_+)=0$, and the fan reward vanishes because
$v_\sigma<0$. Let the unfolded angle advance by $D$ during the dive, jumps included. Then $g=\Theta$,
so $\RF=\max(\RF_0,\Theta)$ and the state moves from $(0,a_0,c_0)$ to $(0,a_1,c_1)$ with
\[
  c_1=\max(c_0-D,0),\qquad a_1\le\min(a_0+D,\tfrac\pi2),
\]
and $V(0,a_1,c_1)\le V(0,\min(a_0+D,\frac\pi2),c_1)$ because $V_a\ge0$. Let $q^{\rm nf}(u)$ be the value
of $q$ at the advance $u\in[0,D]$ if nothing is finalised beyond what the constraint $a\le\pi/2$
forces. Its slope in $u$ is $0$, then $+1$ or $-1$, then $0$, so it is monotone and
$q^{\rm nf}\le\max(q_0,q_1)$; and $q(\tau)\le q^{\rm nf}(u(\tau))$, since $\RF$ does not depend on the
finalisation and $\pfin\ge\max(\pfin(\text{start}),\Theta-\frac\pi2)$. As $Q\le\frac\pi2+\arccos m_*<2\pi$, the
reward of the dive is at most $-(2\pi-Q_+)(\tau_1-\tau_0)$.
If $D=0$, then $c_1=c_0$ and $a_1\le a_0$, so (A1) gives $\Delta W\le0$ and the time reward is
nonpositive. This case needs no instance of (A4). In the remainder assume $D>0$.

By \cref{lem:divetime} the dive lasts at least $\Delta\tau_{\min}(\min(D,\pi))$, and for
$D>\pi$ the arrival state is the one reached with $D=\pi$. So the instance of (A4) at
$\min(D,\pi)$ is exactly the inequality $\Delta W+\text{reward}\le0$ of the dive.
\end{proof}

\section{Proof of the main theorem}\label{sec:proof}

Suppose a path has $\CR<\Csp$; it is competitive at a level $m>m_*$. By \cref{lem:polygonal}
there is a polygonal path which is competitive at a level $m'>m_*$. By \cref{prop:mu,lem:strict},
$\mu_{m_*}(\tau)\ge2\pi+2\kappa\log(m'/m_*)$ for all $\tau$. Times with $v\ge0$ are unbounded: at a new
maximum of $R$ we have $R(t)\ge h_t(n)\ge m't$. Fix such a $\tau_0$ and let $\tau_1\to\infty$ through
such times. The left side of the inequality of \cref{thm:A} grows linearly in $\tau_1$, the right
side is fixed. This contradicts \cref{thm:A} as soon as a function $V$ with (A1)--(A4) exists;
\cref{thm:cert} provides one. With \cref{prop:spiral}, $C_*=\Csp$. \qed

\section{The certificate}\label{sec:cert}

\paragraph{The function.} In the cap-angle coordinate $\beta=h(v)\in[0,\arccos m_*]$ (so
$\mathrm{d}\beta/\mathrm{d}v=\cot\beta$, $P=\cos\beta/m_*$) let
\[
  V(v,a,c)=G(\beta,a,c)+F(v,a),\qquad F(v,a)=\int_{-\beta}^{\clip(a,-\beta,\beta)}\bigl(v+\log\cos s\bigr)\dd s,
\]
so that $F_a=\fan$ and $F_v=\clip(a,-\beta,\beta)+\beta$; $F$ is $C^1$. The function $G$ is a tensor cubic
B-spline on $[0,\frac32]\times[-\frac85,\frac85]\times[0,\frac32]$ with $23\times29\times14$ coefficients and
simple interior knots, all of them double-precision numbers, hence dyadic rationals. It has four
exact structural properties:
\begin{enumerate}[leftmargin=2em,itemsep=1pt]
\item\label{it:corner} \emph{The corner.} $G_\beta(0,a,c)=0$ and $G_c(0,a,0)=0$ (coefficient identities $c_{1jk}=c_{0jk}$, $c_{0j1}=c_{0j0}$).
Hence $V_v=G_\beta\cot\beta+F_v$ is continuous up to $v=0$, $V$ is $C^1$ on $D_3$, and the surface
form (A3), in which $V_v+V_ch'(v)=(G_\beta+G_c)\cot\beta+F_v$, has a limit at the corner.
\item\label{it:mono} \emph{Monotonicity}, that is (A1), that is $G_a\ge0$: by the derivative formula for B-splines \cite{deboor1978}, $G_a$ is a
combination of nonnegative B-splines with coefficients proportional to $c_{i,j,k}-c_{i,j-1,k}$, and these
differences are nonnegative (the smallest one outside the tied block is $1.0006\cdot10^{-7}$).
\item\label{it:flat} A \emph{flat patch}: the $7\times6\times5$ block of coefficients which is active on the knot cells
around the spiral state is constant along $a$, so $G_a\equiv0$ there.
\item\label{it:jet} \emph{The jet correction.} One real number
$\delta_0=-1.7182085816526\ldots\cdot10^{-11}$ is added to one tied group of coefficients, and is
fixed by the requirement $K_\beta(\beta_*)\cot\beta_*+2\pi=\delta_*P_*^2/k_*$, where
$K(\beta,a)=G(\beta,a,\beta)$. This is the one scalar equation which makes \cref{lem:jet} hold
\emph{exactly} and not merely to the accuracy of the fit, and it is needed because the certificate
asserts equality in (A3) at the spiral state: a slack of either sign there, however small, would
make \cref{thm:cert} false or its equality claim wrong. $K_\beta(\beta_*)$ is affine in $\delta_0$, so
the equation is linear and has a unique solution; it is enclosed in a ball of radius $10^{-53}$, and
all enclosures below hold for every number of that ball.
\end{enumerate}
\Cref{app:cert} writes down the $C^1$ extension of $V$ past $\beta=0$ and $c=0$ promised in
\cref{sec:verif}. The saved output of the fit satisfies the identities of items 1 and 3 only up to
$2.7\cdot10^{-9}$; the certified function is the table after three exact projections (521 coefficients
are overwritten by copies of other coefficients), which is again a table of doubles and is stored as
such (\file{output/lifted-three-state-certificate-projected-2026-09-19.json}).

\begin{lemma}[Exact tightness at the spiral]\label{lem:jet}
Let $H(\beta,a)=-\Phi\bigl(V_v+V_c\,h'(v),\,V_a\bigr)$ be the slack of {\rm(A3)} on the surface
$c=h(v)$. At the spiral state $(\beta_*,\alpha_*)$ the value and
the gradient of $H$ vanish; the two equations \eqref{eq:contact} and \eqref{eq:pitch} of
\cref{def:constants} are exactly what makes this happen.
\end{lemma}

The short computation is in \cref{app:cert}. The vanishing is also necessary: the spiral is a closed
orbit of the relaxed system with zero mean reward, so every valid storage has equality along it.

\begin{theorem}[Certificate]\label{thm:cert}
The function $V$ satisfies {\rm(A1)--(A4)}, with equality in {\rm(A3)} only at the spiral
state and strict inequality in {\rm(A2)}.
\end{theorem}

\begin{proof}[Computer-assisted proof]
All computations use Arb ball arithmetic \cite{johansson2017arb} through python-flint.

\emph{Constants.} An interval Newton step for \eqref{eq:contact}--\eqref{eq:pitch}, with the Jacobian
obtained by automatic differentiation, proves that $(k_*,m_*)$ exists and is unique in the box of
\cref{def:constants} (\file{scripts/check_spiral_constants_box.py}). Each implementation then pins it
down in its own way. The second refines a numerical solution by twelve ordinary Newton steps at
$400$ bits and applies an interval Newton step to the box of radius $10^{-90}$ about it: the step
returns an enclosure strictly inside that box, so the root lies there, and by the uniqueness just
proved it is $(k_*,m_*)$. The first uses a Krawczyk operator with a point inverse, in the
coordinates $(\alpha,\beta)$ rather than $(k,m)$ and on a box of radius $10^{-65}$ about a stored
midpoint. What is propagated below has radius $10^{-55}$ in the first implementation and below
$10^{-118}$ in the second, at every working precision.

\emph{The region covered.} In the coordinate $\beta$ the reachable state set of \cref{lem:state} is
\[
  \hat D_3=\{(\beta,a,c):\ 0\le\beta\le c\le\arccos m_*,\ |a|\le\tfrac\pi2\},
\]
the bound $v\le\log\frac1{m_*}$ of $D_3$ being the same as $c\le\arccos m_*$ on the surface. Both
implementations run their covers over exactly this set. That the cells they keep tile $\hat D_3$
with no gap, and that the dive slabs tile $(0,\pi]$, is checked separately
(\file{scripts/check_verifier_domain_2026_09_20.py}); a gap there would leave the theorem unproved
without making any checked inequality false, so that check re-derives the filters from the knot
vectors instead of importing them.

\emph{Covers.} On a knot cell $G$ and its partial derivatives are polynomials with exact rational
(or, in the corrected group, ball) coefficients. On a box, a lower bound for a slack $S$ is
$S(x_0)-\sum_i\sup|\partial_iS|\,r_i$, with the partial derivatives enclosed over the box. At the
kinks --- of $F_v$, of $\fan$, of $(c+a)_+$, of the positive part inside $\Phi$, of the norm at the
origin and of the caps in (A4) --- the enclosures are hulls over the branches; this encloses
Clarke's generalized gradient, and the mean value inequality is Lebourg's \cite{clarke1983}. Boxes are
bisected until the bound is positive; the two implementations name the slacks of (A2), (A3) and
(A4) \texttt{H3}, \texttt{H4} and \texttt{T}. \Cref{app:cert} describes the treatment of the surface
$c=\beta$, of the corner $\beta=0$ and of short dives.

\emph{The spiral.} On a box $N$ around $(\beta_*,\alpha_*)$ inside the flat patch and the smooth
branch, the Hessian of $H$, enclosed with second-order automatic differentiation, is positive
definite; with \cref{lem:jet} and convexity of $N$, $H\ge0$ on $N$ with equality only at the spiral
state. Outside $N$ the surface is covered as above. \Cref{tab:cover} gives the counts.
\end{proof}

\begin{table}[t]
\centering\small
\begin{tabular}{@{}lrrrr@{}}
\toprule
 & \multicolumn{2}{c}{first implementation} & \multicolumn{2}{c}{second implementation}\\
\cmidrule(lr){2-3}\cmidrule(l){4-5}
hypothesis & boxes & least bound & boxes & least bound\\
\midrule
(A3) surface, outside $N$ & 365{,}466 & $9.1\cdot10^{-15}$ & 259{,}560 & $1.0\cdot10^{-13}$\\
(A2) all headings & 816{,}750 & $6.1\cdot10^{-7}$ & 550{,}362 & $4.8\cdot10^{-7}$\\
(A4) small dives & 650 & $5.5\cdot10^{-4}$ & 4{,}284 & $3.4\cdot10^{-5}$\\
(A4) other dives & 15{,}224 & $6.1\cdot10^{-7}$ & 11{,}154 & $1.6\cdot10^{-7}$\\
half-widths of $N$ & \multicolumn{2}{r}{$7.3\cdot10^{-6}\times9.8\cdot10^{-6}$} & \multicolumn{2}{r}{$2^{-13}\times2^{-13}$}\\
Hessian of $H$ on $N$ & \multicolumn{2}{r}{$\bigl[\begin{smallmatrix}38\pm0.75&-0.2\pm0.04\\&0.154\pm0.0004\end{smallmatrix}\bigr]$}
 & \multicolumn{2}{r}{$\bigl[\begin{smallmatrix}40\pm3.8&0\pm0.31\\&0.15\pm0.007\end{smallmatrix}\bigr]$}\\
\bottomrule
\end{tabular}
\caption{The two verifications of \cref{thm:cert}. Box counts are identical at 96, 128 and 192 bits.
The second implementation imports nothing from the first: own constants, spline pieces by exact
interpolation of Cox--de Boor values, automatic differentiation instead of hand-derived gradients,
its own treatment of $\beta=0$, of $N$ and of small dives. Point values of the three slacks
agree to $10^{-33}$, $\delta_0$ to 40 digits.}
\label{tab:cover}
\end{table}

\begin{figure}[t]
\centering
\includegraphics[width=0.68\linewidth]{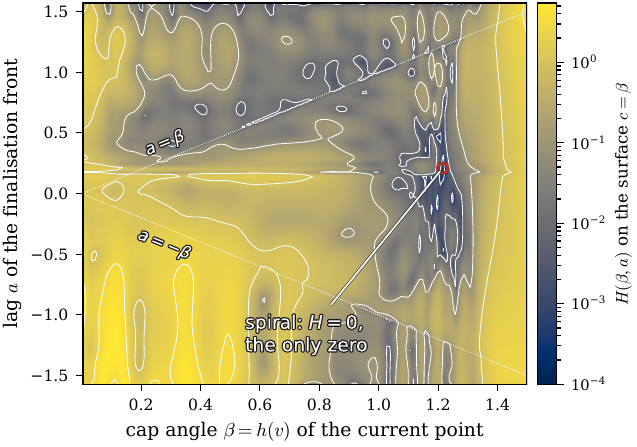}\\[4pt]
\includegraphics[width=0.86\linewidth]{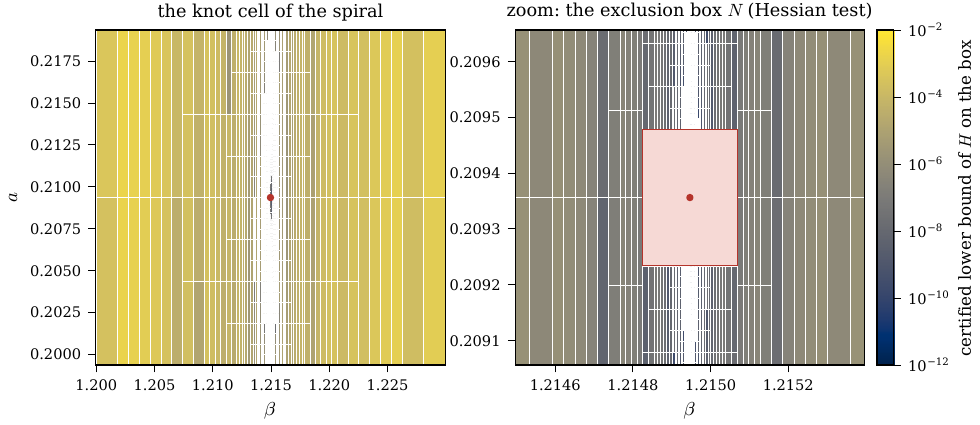}
\caption{Top: the slack $H$ of hypothesis (A3) on the surface $c=\beta$ (floating-point
evaluation, logarithmic colour scale). It vanishes only at the spiral state. Bottom: the certified
boxes of the second implementation on the knot cell of the spiral, and the exclusion box $N$ on
which the exact jet and the interval Hessian are used.}
\label{fig:H4}
\end{figure}

\section{What is checked, and what is not}\label{sec:checks}

The proof consists of (i) the prose reduction of
\cref{sec:polygonal,sec:lift,sec:book,sec:traj,sec:verif}, (ii) Lean-checked lemmas, and (iii) the
certificate.

\paragraph{Lean.} Eight modules (Lean~4 v4.33.0, mathlib \cite{mathlib2020} at \texttt{db584cd6})
without \texttt{sorry}; the only axioms are \texttt{propext}, \texttt{Classical.choice} and
\texttt{Quot.sound}. The compiled modules were then replayed through the kernel with
\texttt{leanchecker} (\file{scripts/check-lean-kernel-replay.sh}), which re-accepted every
declaration. That is the same kernel run again on the build artefacts rather than a second
implementation of it; what it adds to the axiom list is that no declaration reached an
\texttt{.olean} without the kernel having accepted it.

\begin{center}\small
\begin{tabular}{@{}>{\raggedright\arraybackslash}p{0.34\linewidth}p{0.6\linewidth}@{}}
\toprule
module & statements used above\\
\midrule
\lean{PolygonalReduction} & \cref{lem:polygonal}\\
\lean{KneserPoulsenLine}, \lean{LiftedCoverage}, \lean{AngularUnfolding} & \cref{lem:vertices,lem:KP} and the unfolding\\
\lean{MonotoneSelector} & concavity of $\log\cos$, exchange, angle part of \cref{lem:selector}\\
\lean{LiftedBookkeeping} & per-direction bound, widening, $h'\ge\kappa$, dive ratio\\
\lean{RunningMaxFront} & the two facts behind \eqref{eq:S}\\
\lean{LiftedHeading} & supremum over headings in (A2), (A3)\\
\bottomrule
\end{tabular}
\end{center}
Not formalized: the construction of the polygonal path from its vertex times; Tonelli's theorem
and the push-forward $\mathrm{d}\varphi\mapsto\mathrm{d}\pfin$ in \cref{thm:book}; existence of
maximisers (\cref{lem:attained}); the absolute continuity statements of \cref{lem:skorokhod}; the
chain rule in \cref{thm:A}; the planar length bound for dives (\cref{lem:divetime}).

\paragraph{Tests.} All quantities of \cref{thm:book,thm:A} were recomputed from the geometry of 73
polygonal paths with 101 dives --- perturbed spirals with chords, radial spikes, runs along the rim
$R=(1+\epsilon)m_*t$, backward retracing, vertices at and departures from the origin, waiting, other
pitches. Every inequality of the reduction held, $V+\text{reward}$ never increased (\cref{fig:W}),
and the state-space bounds of \cref{lem:state} were \emph{measured} rather than imposed.
\Cref{app:test} reports this in detail. The test is in floating point and is not part of the proof.

\paragraph{The audit.} An internal audit rederived (i), compared every cited Lean
statement with its use, wrote the second implementation of (iii), and ran the geometric test; it
found one wrong sentence (on dives through a vertex at the origin), repaired here by the planar
length argument, and a few minor inaccuracies, all corrected. The subsequent verification review reproduced both
certificates and the Lean build, and prompted the explicit source domain $\Omega$, the
initialization of $\pfin$ at $L$, the push-forward identity \eqref{eq:pushforward}, and the endpoint
and atom conventions in \eqref{eq:reward} and \eqref{eq:fan-atom}; these changes leave the
certificate inequalities unchanged. The trusted base of (iii) is Arb, python-flint and CPython.

\paragraph{Reproduction.} Everything needed to rerun the computer-assisted part is included
with this submission as ancillary files: the certified coefficient table (sha256
\texttt{99c016ec\dots 4ccac6}), both verifiers, the fit, the tests and the eight Lean modules, with
the Python and mathlib versions pinned. The file names quoted in this paper are paths inside that
directory (\file{anc/} on arXiv), and its \file{README.md} lists the commands. The two
verifications take about ninety and sixty seconds on a laptop.

\appendix

\section{The lift: construction, attainment, and the domain of directions}\label{app:lift}

\paragraph{Construction of the lift.} Along a polygonal path choose a lift $\tilde\theta(s)$ of the
polar angle which is continuous while $R>0$; at a passage through the origin it jumps, and the jump
is chosen of absolute value at most $\pi$. During a wait at the origin keep the incoming angle fixed
and put any jump at departure; before the first departure use its outgoing angle. Choose the
right-continuous representative. On an edge which misses the origin the angle is monotone and turns
by less than $\pi$, so $\tilde\theta$ has locally bounded variation on $(0,\infty)$ and its signed
variation function $\Theta$ of \cref{sec:lift} is well defined.

\begin{lemma}[Attainment and the tie-breaking]\label{lem:attained}
For $\varphi\in\Omega$ the supremum defining $E(\varphi)$ is attained, the set $M$ of maximisers is
nonempty and compact, and the final source $\sigma(\varphi)$ --- the maximiser with the largest
angle, earliest among those --- exists.
\end{lemma}
\begin{proof}
A fixed finite lower bound on the score bounds $R$ away from zero, hence bounds physical source
time away from zero by $R(s)\le s$; on the resulting compact set the score is upper
semicontinuous, because $z$ is continuous where $R>0$ and, by \cref{lem:jumps}, the jumps of
$\Theta$ all occur where $R=0$ and the score is $-\infty$. So $M$ is nonempty and compact. Since
$\Theta$ is nondecreasing, the largest angle on $M$ is $\theta_{\max}=\Theta(\max M)$; the set
$S=\{s\in M:\Theta(s)=\theta_{\max}\}$ is nonempty, and if $s_n\downarrow\inf S$ in $S$ then
$\inf S\in M$ because $M$ is closed and $\Theta(\inf S)=\lim\Theta(s_n)=\theta_{\max}$ because
$\Theta$ is right-continuous; so $\min S$ exists. (The set $S$ itself need not be closed: $\Theta$
may jump up at $\sup S$. Right-continuity, and not continuity, is what is used.)
\end{proof}

\begin{lemma}[The domain is an interval up to a countable set]\label{lem:omega}
$\Omega$ is the union of the open intervals $(\Theta(s)-\pi/2,\Theta(s)+\pi/2)$ over the times
$s\le\tau_1$ with $R(s)>0$, and $(L,U)\setminus\Omega$ is countable. If $R(e^{\tau_1})>0$ then
$U=\Theta(\tau_1)+\pi/2$.
\end{lemma}
\begin{proof}
The description of $\Omega$ is the definition of the score. Let $J=\{s\le\tau_1:R(s)>0\}$, which is
relatively open. Its complement is a union of origin passages and waits; a polygonal path has
finitely many edges in each compact subinterval of $(0,\infty)$ and each edge meets the origin in
one point or in a subinterval, so the complement, and hence the set of connected components of
$J$, is countable. On one component $\Theta$ is continuous and nondecreasing, so the union of the
intervals over that component is again an open interval. Let two components be adjacent, with
limiting angles $\Theta^-\le\Theta^+$; the angle advance across an origin passage or wait is at
most $\pi$, so $\Theta^+-\Theta^-\le\pi$ and the two open intervals
$(\Theta^\mp-\pi/2,\Theta^\mp+\pi/2)$ miss at most the single point $\Theta^-+\pi/2$, and only when
the advance is exactly $\pi$. Countably many components give a countable exceptional set. Finally
every source has $\Theta(s)\le\Theta(\tau_1)$, so $\Omega\subseteq(-\infty,\Theta(\tau_1)+\pi/2)$, with
equality of the suprema when $R(e^{\tau_1})>0$: the half-line above $U$ has no source at all.
\end{proof}

\noindent Polygonality enters \cref{lem:omega} only through the countability of the origin passages.

\begin{lemma}[Measurability of the covered set]\label{lem:measurable}
$\{(\varphi,\tau):\varphi\text{ is covered at }\tau\}=\{E_\tau(\varphi)\ge\tau+\log m\}$ is measurable.
\end{lemma}
\begin{proof}
Both $z$ and $\Theta$ are right-continuous in $s$, so in
$E_\tau(\varphi)=\sup_{s\le\tau}[z(s)+\log\cos(\varphi-\Theta(s))]$ the supremum may be taken over
$(D\cap[0,\tau))\cup\{\tau\}$ for a fixed countable dense $D$. Then $E$ is a countable supremum of
jointly measurable functions.
\end{proof}

\section{Regularity of the state}\label{app:reg}

\begin{proof}[Proof of \cref{lem:state}]
Let $\varphi\in\Omega$ with $\varphi<\Theta(\tau)-\pi/2$. Every source $s$ of $\varphi$ has
$|\varphi-\Theta(s)|<\pi/2$, hence $\Theta(s)<\varphi+\pi/2<\Theta(\tau)$, hence $s<\tau$ because
$\Theta$ is nondecreasing: $\varphi$ has only past sources, so $\tau_\sigma(\varphi)<\tau$ and
$\varphi\le\pfin(\tau)$. If $L<\Theta(\tau)-\pi/2$ then by \cref{lem:omega} such $\varphi$ come
arbitrarily close to $\Theta(\tau)-\pi/2$ from below, and if $L\ge\Theta(\tau)-\pi/2$ the
initialization gives $\pfin(\tau)\ge L$ directly; either way
$\pfin(\tau)\ge\Theta(\tau)-\pi/2$. Conversely a finalised direction has a source with angle at
most $\Theta(\tau)$, so $\pfin(\tau)\le\Theta(\tau)+\pi/2$; and $L\le\Theta(\tau)+\pi/2$, so the
initialization does not break this.

Nothing in that argument used $R(\tau)>0$, and it is worth saying why the zero-radius times --- the
origin passages and waits, which carry the jumps of $\Theta$ --- need no separate treatment. At such
a $\tau$ there is no source at all, so no direction is finalised there: $\dd\pfin$ has no atom at
$\tau$ and $\pfin(\tau)=\pfin(\tau-)$. The bound $\pfin(\tau)\ge\Theta(\tau)-\pi/2$ with the
\emph{post-jump} angle then says that the directions below $\Theta(\tau)-\pi/2$ were already
finalised strictly before $\tau$, which is exactly what the first paragraph proves, since every
$s\ge\tau$ has $\Theta(s)\ge\Theta(\tau)$. The advance across the passage is at most $\pi$, so
$\Theta(\tau)-\pi/2$ is still below $U$ and the argument has room; $a$ reaches $\pi/2$ there and
does not exceed it. Before the first departure $\Omega$ starts at $L=\Theta-\pi/2$, the front is
$L$, and $a=\pi/2$.

For $c$: the term $s=\tau$ gives $c\ge h(v_+)$, and
$c=\sup_{s\le\tau}[\Theta(s)-\Theta(\tau)+h(v(s)_+)]\le h(\log\frac1m)$.
\end{proof}

\begin{proof}[Proof of \cref{lem:skorokhod}]
(1) On an edge, $R^2$ and $(m t)^2$ are quadratic in $t$, so $v$ has at most two zeros or vanishes
identically. (2) On each of the finitely many pieces on which $v$ is monotone, $h(v_+)$ is the
composition of the absolutely continuous monotone function $h$ with a monotone absolutely
continuous function. The finiteness is essential and not decoration: $h'(0^+)=\infty$, so $h$ is
not Lipschitz at $0$ and $h\circ v$ need not even be of bounded variation for a general $v$ of
bounded variation --- for $v$ oscillating $n$ times between $0$ and $1/n$, $\Var(v)$ stays
bounded while $\Var(h\circ v)$ grows like $\sqrt n$, since $h(u)\sim\sqrt{2u}$ at $0$.
Part~(1) is what excludes this. For $s<t$,
$0\le\RF(t)-\RF(s)\le\sup_{s\le u\le t}(g(u)-g(s))_+\le\Var(g;[s,t])$. (3) On the open set
$\{\RF>g\}$ the running maximum is locally constant (\lean{Shore.running_max_const}), so $c'=-\Theta'$.
Two functions which agree on a set have equal derivatives at every accumulation point of the
set at which both are differentiable, and the other points of the set are countable
(\lean{Shore.contact_deriv_ae}); so $\RF'=g'$ almost everywhere on the contact set $\{\RF=g\}$, where
$c=h(v)$. (4) $\pfin$ is monotone and, by \cref{lem:jumps}, $\Theta$ is continuous outside dives, so
$a=\Theta-\pfin$ can jump only downwards there.
\end{proof}

\begin{lemma}[Dive time]\label{lem:divetime}
Let a dive run from $t_0$ to $t_1$ and let the unfolded angle advance by $D>0$ during it, jumps
included. Then $\tau_1-\tau_0\ge\Delta\tau_{\min}(\min(D,\pi))$. If $D>\pi$, the state reached at
the end of the dive is the one reached with the advance $\pi$.
\end{lemma}
\begin{proof}
The curve $\gamma=Re^{i\Theta}$ is a continuous planar curve of the same length as the dive:
$|\Theta'|=|\tilde\theta'|$ away from the origin, and $\Theta$ jumps only where $R=0$
(\cref{lem:jumps}), where $\gamma$ vanishes on both sides, so $\gamma$ is continuous across the
jump and $|\gamma'|^2=R'^2+R^2\Theta'^2=|x'|^2\le1$ elsewhere. The dive starts at $R_0=m_*t_0$ and
ends at $R_1=m_*t_1$. If $D\le\pi$, passages through the origin included,
\[
 t_1-t_0\ \ge\ \text{length}\ \ge\ |\gamma(\text{end})-\gamma(\text{start})|=m_*\sqrt{t_0^2+t_1^2-2t_0t_1\cos D}.
\]
If $D>\pi$, let $B$, of radius $r\ge0$, be the point of $\gamma$ at which the advance reaches $\pi$ (the
origin, if that value is jumped over). Then the length is at least $(R_0+r)+|R_1-r|\ge R_0+R_1$, which is
the value of the chord for $D=\pi$, and the arrival state is that of $D=\pi$ because $a_0\ge-\frac\pi2$
and $c_0<\pi$. With $x=t_1/t_0$ and $1+x^2-2x\cos D=(x-1)^2+4x\sin^2\frac D2$, the inequality
$x-1\ge m_*\sqrt{1+x^2-2x\cos D}$ reads $\sqrt x-1/\sqrt x\ge2\kappa\sin\frac D2$, that is
$\tau_1-\tau_0\ge\Delta\tau_{\min}(\min(D,\pi))$. (The function $x\mapsto(x-1)-m\sqrt{1+x^2-2x\cos D}$
increases with slope at least $1-m$, \lean{Shore.dive_ratio_ge_root}.)
\end{proof}

\section{The certificate in detail}\label{app:cert}

\paragraph{The $C^1$ extension.} In $\beta$ rather than $v$, the storage is
$\hat V(\beta,a,c)=G(\beta,a,c)+\hat F(\beta,a)$ with
$\hat F(\beta,a)=\int_{-\beta}^{\clip(a,-\beta,\beta)}\log\frac{\cos s}{\cos\beta}\dd s$, so that
$\hat F_a=\fan$ and $\hat F_\beta=\bigl(\clip(a,-\beta,\beta)+\beta\bigr)\tan\beta$: both are continuous
for $|\beta|<\frac\pi2$, so $\hat F$ is $C^1$ there. A cubic spline with simple knots is $C^2$ on its box,
and prolonging the polynomials of the end cells extends $G$ past $\beta=0$ and $c=0$ with the same
smoothness; since $\arccos m_*<\frac32$ and $\frac\pi2<\frac85$, the image of $D_3$ lies in the box.
Hence $\hat V$ is $C^1$ on an open neighbourhood of that image, and $V=\hat V(h(v),a,c)$ is the
restriction of a $C^1$ function, as \cref{sec:verif} requires. The singular factor
$h'(v)=\cot\beta$ occurs only in $\dd\beta=\cot\beta\dd v$, never in $\hat V$, and
$V_v\dd v=\hat V_\beta\dd\beta$: \cref{it:corner} of \cref{sec:cert} says exactly that $\hat V_\beta$ vanishes on
$\beta=0$, which is what makes $\beta=h(v_+)$ enter the chain rule of \cref{thm:A} as an absolutely
continuous coordinate (\cref{lem:skorokhod}(2)).

\begin{proof}[Proof of \cref{lem:jet}]
Near the spiral state we are on the flat patch and on the branch $|a|<\beta$, $a+\beta>0$, where
$V_a=\delta:=\log(\cos a/\cos\beta)$ and $F_v=a+\beta$. With $A=2\pi-a-\beta$ and
$\tilde c(\beta)=(K_\beta\cot\beta+2\pi)/P$, which depends on $\beta$ only,
\[
  H=P\bigl[\tilde c-N\bigr],\qquad N=\sqrt{X^2+\delta^2},\qquad X=P\tilde c-A .
\]
The jet correction (\cref{it:jet} of \cref{sec:cert}) is exactly the statement
$\tilde c(\beta_*)=\delta_*P_*/k_*$. \emph{Value:} since
$P_*^2=1+k_*^2$, we have $N=\tilde c$ as soon as $X=\delta/k$; and $X=\delta P^2/k-A$ equals $\delta/k$
iff $k\delta=A$, the contact equation \eqref{eq:contact}. Then $X/N=1/P$ and $\delta/N=k/P$, the
heading of the spiral. \emph{$a$-derivative:} $\partial_aH=-P(X-\delta\tan a)/N$, which vanishes
because $X=\delta/k$ and $\tan\alpha_*=1/k_*$. \emph{$\beta$-derivative:} the factor of $\tilde c'$ is
$P(1-XP/N)=0$, the term $P_\beta(\tilde c-N)$ vanishes, and what remains is
$-(P/N)[X(1-P\tilde c\tan\beta)+\delta\tan\beta]$, which vanishes iff $\delta(1+k^2)=k(k+\cot\beta)$,
the optimal-pitch equation \eqref{eq:pitch}.
\end{proof}

\paragraph{Boundaries of the cover.} Both implementations keep a knot cell unless it lies beyond
$\arccos m_*$ in $\beta$ or $c$ or beyond $\frac\pi2$ in $|a|$, clip the straddling cell to the
bound, and discard a $(\beta,c)$ pair only when the whole cell lies below the surface $c=\beta$.
Since $\arccos m_*=1.49833\ldots<\frac32$ and $\frac\pi2<\frac85$, the top slivers of the spline box
are never visited, correctly: they are not reachable. A box which is cut by the surface $c=\beta$ is
expanded about its corner $(\beta_{\min},\cdot,c_{\max})$: the segment from that corner to a point of
the box with $c\ge\beta$ stays in $\{c\ge\beta\}$, because $c-\beta$ is affine on it and nonnegative
at both ends. Near $\beta=0$ one uses that $G_\beta/\beta$ and $(G_\beta+G_c)(\beta,a,\beta)/\beta$ are
polynomials, by \cref{it:corner} of \cref{sec:cert}, and that $\beta\cot\beta$ is smooth.

\paragraph{Dives.} Write $T(a,c,D)$ for the slack of (A4). Then $T(a,c,0)=0$, so short dives need an
argument beyond a cover. Either of two is used: the generalized $D$-derivative of $T$ is positive on
$[0,\frac18]$ (second implementation), or $T$ is bounded below through
$\Gamma(a,c)=(2\pi-(a+c)_+)\kappa-(G_a-G_c)(0,a,c)$, the concavity of $\Delta\tau_{\min}$, (A1) and
$G_c(0,a,0)=0$ (first implementation, $D\le\frac1{16}$). Larger dives are covered directly in
$(a,c,D)$.

\paragraph{How $V$ was found.} (A2) and (A3) are second-order cone constraints on the
coefficients, (A1) and (A4) are linear. A cone program maximises a margin which is weighted so as
to vanish at the spiral state. Dense floating-point grids are \emph{not} a reliable test here ($P$
reaches $13.8$ near $v=0$, and the constraint functions have valleys a few thousandths wide): a
candidate which passed three dense grids was rejected by the exact covers. The fit therefore uses
the covers themselves as a separation oracle. At the level $m_*$ the optimal margin is positive; in
a uniform-margin calibration it is $-0.155$, $-0.051$, $+3.4\cdot10^{-7}$, $+0.051$ at
$m/m_*=0.97,0.99,1,1.01$, so the relaxation loses nothing visible. The verification does not
depend on how $V$ was found, and the construction is reproducible: \file{reproduce-certificate.sh}
runs the fit with its oracle, the projection and both verifiers from scratch. On a laptop this
takes about seventy-five minutes and produces a different spline, with the same optimal margin,
which both implementations verify.

\section{The geometric test}\label{app:test}
\begin{figure}[t]
\centering
\includegraphics[width=0.92\linewidth]{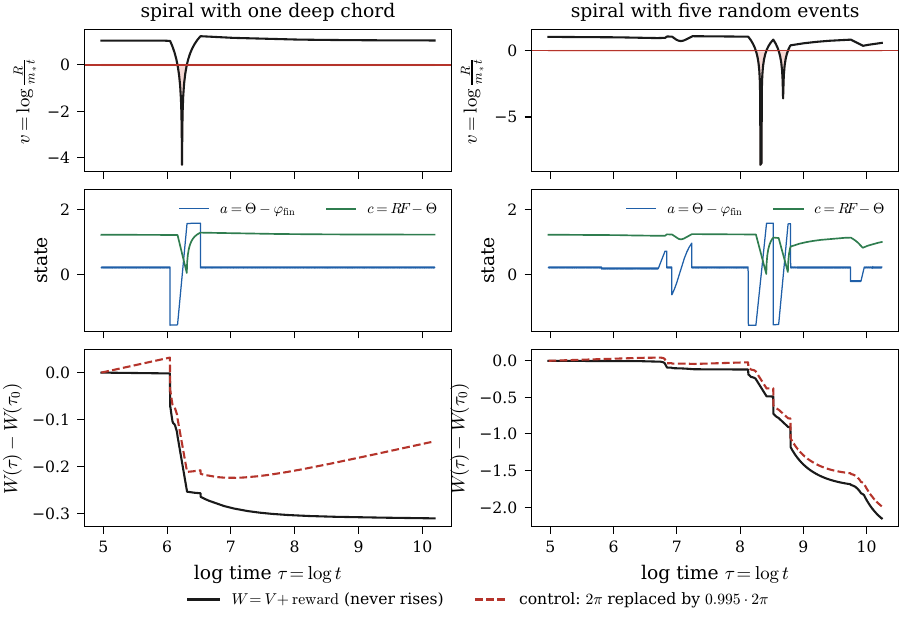}
\caption{Two test paths with decreasing radius. Top: $v<0$ is a dive below the guaranteed disk.
Middle: the state, computed from the geometry. Bottom: $V+\text{accumulated reward}$; the dashed
line is a control with a deliberately wrong accounting.}
\label{fig:W}
\end{figure}

As a safeguard against a modelling error in \cref{thm:book,thm:A}, all quantities were computed
from the geometry of 73 polygonal paths (perturbed spirals with chords, radial spikes, overshoots,
runs along the rim $R=(1+\epsilon)m_*t$ with $|\epsilon|$ down to $10^{-3}$, backward retracing,
vertices at the origin, departures from the origin at exactly the antipodal angle, waiting, other
pitches; 101 dives): the exact covered time of every lifted direction, the final sources, both
fronts and the state. \Cref{lem:selector} held exactly; the two inequalities behind \cref{lem:KP} up
to rounding; the front bound in the proof of \cref{thm:book} up to half a grid step; the dive-time
bound with straight chords as the equality case; and $V+\text{reward}$ never increased, at three
resolutions (\cref{fig:W}).

The state-space bounds of \cref{lem:state} are \emph{measured} there and not imposed: the front is
recomputed without the clamp at $\Theta-\pi/2$ which would enforce $a\le\pi/2$, and the excursion of
$(v,a,c)$ outside $D_3$ is reported face by face. It is zero on $v$ and on $c\le\arccos m_*$, is
$7\cdot10^{-15}$ on $c\ge h(v)$, and on $|a|\le\frac\pi2$ it is one grid step, which is the front
being a cell count rather than an escape. The sharp case is a departure from the origin at exactly
the antipodal angle: there the unfolded advance is exactly $\pi$, $\Omega$ really does lose a point,
and $a$ touches $\frac\pi2$ with no slack, so any discretisation shows. On those paths the measured
violation of $a\le\frac\pi2$ is $0.39$, $0$ and $0.19$ grid steps at grid $0.002$, $0.001$ and
$0.0005$: below one cell at every resolution. This matters because \cref{lem:state} is the
hypothesis under which the certified inequalities are about the real path at all, and it is the one
statement of the reduction whose first version was false.

A half per cent error in the level $2\pi$, or one per cent in the fan reward, makes $V+\text{reward}$
increase visibly on every path. The test is in floating point and is not part of the proof; it
would not detect a second-order modelling error below its discretisation drift.

\bibliographystyle{plain}
\bibliography{references}
\end{document}